\documentclass[11pt,reqno]{amsart}

\usepackage{amssymb}
\usepackage{amsmath,amsthm,amsfonts,amssymb,latexsym,mathrsfs,color,hyperref}

\newtheorem{theorem}{Theorem}
\newtheorem{corollary}[theorem]{Corollary}
\newtheorem{proposition}[theorem]{Proposition}

\newtheorem{lemma}[theorem]{Lemma}
\newtheorem{definition}[theorem]{Definition}
\newtheorem{example}[theorem]{Example}

\newcommand{\md}{\mathcal{D}}

\newcommand{\lrmin}{{\rm lrmin\,}}

\newcommand{\I}{{\rm I\,}}

\newcommand{\rlmin}{{\rm rlmin\,}}

\newcommand{\drop}{{\rm drop\,}}

\newcommand{\ap}{{\rm ap\,}}
\newcommand{\lap}{{\rm lap\,}}

\newcommand{\des}{{\rm des\,}}

\newcommand{\exc}{{\rm exc\,}}

\newcommand{\cyc}{{\rm cyc\,}}

\newcommand{\fix}{{\rm fix\,}}

\newcommand{\msn}{\mathfrak{S}_n}

\newcommand{\ms}{\mathfrak{S}}

\newcommand{\lrf}[1]{\lfloor #1\rfloor}

\newcommand{\mbn}{{\mathcal B}_n}
\newcommand{\mq}{\mathcal{Q}}
\newcommand{\mqn}{\mathcal{Q}_n}
\newcommand{\z}{ \mathbb{Z}}
\newcommand{\asc}{{\rm asc\,}}

\title{Symmetric decompositions and Euler-Stirling statistics on Stirling permutations}
\author[S.-M.~Ma]{Shi-Mei Ma}
\address{School of Mathematics and Statistics, Shandong University of Technology, Zibo 255000, Shandong, P.R. China}
\email{shimeimapapers@163.com (S.-M. Ma)}
\author[J. Wang]{Jianfeng Wang}
\address{School of Mathematics and Statistics, Shandong University of Technology, Zibo 255000, Shandong, P.R. China}
\email{jfwang@aliyun.com (J. Wang)}
\author[G. Yan]{Guiying Yan}
\address{Academy of Mathematics and Systems Science, Chinese Academy of Sciences, Beijing 100190, P.R. China}
\email{yangy@amss.ac.cn (G. Yan)}
\author[J. Yeh]{Jean Yeh}
\address{Department of Mathematics, National Kaohsiung Normal University, Kaohsiung 82444, Taiwan}
\email{chunchenyeh@nknu.edu.tw (J. Yeh)}
\author[Y.-N. Yeh]{Yeong-Nan Yeh}
\address{College of Mathematics and Physics, Wenzhou University, Wenzhou 325035, P.R. China}
\email{mayeh@math.sinica.edu.tw (Y.-N. Yeh)}
\begin{document}

\maketitle
\begin{abstract}
The Stirling permutations introduced by Gessel-Stanley have recently received considerable attention.
Motivated by Ji's work on $(\alpha,\beta)$-Eulerian polynomials (Sci China Math., 2025) and Yan-Yang-Lin's work on $1/k$-Eulerian polynomials (J. Combin. Theory Ser. A, 2026), 
we present several symmetric decompositions of the enumerators related to
Euler-Stirling statistics on Stirling permutations. Firstly, we provide a partial symmetric decomposition for the $1/k$-Eulerian polynomial.
Secondly, we give several unexpected applications of the $(p,q)$-Eulerian polynomials, 
where $p$ marks the number of fixed points of permutations and $q$ marks that of cycles.
From this paper, one can see that $(p,q)$-Eulerian polynomial contains a great deal of information about permutations and Stirling permutations.
Using the change of grammars, we show that the $(\alpha,\beta)$-Eulerian polynomials 
introduced by Carlitz-Scoville can be deduced from the
$(p,q)$-Eulerian polynomials by special parametrizations.
We then introduce proper and improper ascent-plateau statistics on Stirling permutations.
Moreover, we introduce proper ascent, improper ascent, proper descent and improper descent statistics on permutations.
Furthermore, we consider the joint distributions of Euler-Stirling statistics on permutations, including the numbers of 
improper ascents, proper ascents, left-to-right minima and right-to-left minina.
In the final part, we first give a symmetric decomposition of the joint distribution of the ascent-plateau and left ascent-plateau statistics,
and then we show that the 
$q$-ascent-plateau polynomials are bi-$\gamma$-positive, where $q$ marks the number of left-to-right minima. 
\bigskip

\end{abstract}
\date{\today}
\tableofcontents
\section{Introduction}
Let $\msn$ denote the set of permutations on $[n]=\{1,2,\ldots,n\}$.
For $\pi=\pi(1)\cdots\pi(n)\in\msn$, an entry $\pi(i)$ is called a {\it descent} (resp.~{\it ascent},~{\it excedance},~{\it drop},~{\it fixed point})
if $\pi(i)>\pi(i+1)$ (resp.~$\pi(i-1)<\pi(i)$, ~$\pi(i)>i$,~$\pi(i)<i$,~$\pi(i)=i$), where we set $\pi(0)=\pi(n+1)=0$. 
A {\it left-to-right minimum} of $\pi$ is a value $\pi(i)$ such that $\pi(i)<\pi(j)$ for all $j<i$ or $i=1$.
Let $\des(\pi)$ (resp.~$\asc(\pi)$, $\exc(\pi)$, $\drop(\pi)$, $\fix(\pi)$, $\cyc(\pi)$ and $\lrmin(\pi)$) be the number of descents (resp.~ascents, excedances, drops, fixed points, cycles and left-to-right minima) of $\pi$. 
The classical {\it Eulerian polynomials} can be defined by
$$A_n(x)=\sum_{\pi\in\msn}x^{\des(\pi)}=\sum_{\pi\in\msn}x^{\asc(\pi)}=\sum_{\pi\in\msn}x^{\exc(\pi)+1}=\sum_{\pi\in\msn}x^{\drop(\pi)+1}.$$
It is well known that 
$$\sum_{\pi\in\msn}x^{\cyc(\pi)}=\sum_{\pi\in\msn}x^{\lrmin(\pi)}=x(x+1)(x+2)\cdots (x+n-1).$$
As usual, a statistic is called a {\it Eulerian} (resp.~{\it Stirling}) statistic, if it has the same distribution as ``\exc'' (resp.~``\cyc'') on the symmetric group $\msn$.

There has been an increasing interest in the joint distributions of Euler-Stirling statistics on permutations. 
For example, Jin~\cite{Jin23} considered the joint distribution of the numbers of descents, inverse descents, left-to-right maxima and right-to-left maxima. Pei-Zeng~\cite{Pei24} considered the enumeration problem of derangements with signed right-to-left minima and excedances.
Very recently, Ji~\cite{Ji25} 
investigated the joint distribution of descent-Stirling statistics of permutations. More importantly, Ji~\cite{Ji25} discovered a novel way to deduce 
exponential formula. Subsequently, Ji-Lin~\cite{Ji24} 
considered a unified extension of the binomial-Stirling Eulerian polynomials and 
Stirling-Eulerian polynomials. In~\cite{Dong24}, 
Dong-Lin-Pan proved a $q$-analog of Ji's exponential formula. 
Xu-Zeng~\cite{Xu25} 
provided an alternative approach to Ji's exponential formula and gave a generalization of it. Chen-Fu~\cite{Chen25} gave a novel labeling scheme for increasing binary trees as an alternative interpretation
of the $(\alpha,\beta)$-Eulerian polynomials introduced by Carlitz-Scoville~\cite{Carlitz74}.
Motivated by Ji's work on $(\alpha,\beta)$-Eulerian polynomials~\cite{Ji25} and Yan-Yang-Lin's work on $1/k$-Eulerian polynomials~\cite{Yan26},
the objective of this paper is to investigate the polynomials involving Euler-Stirling statistics on Stirling permutations. 

Following Savage-Viswanathan~\cite{Savage2012}, 
the {\it $1/k$-Eulerian polynomials} $A_n^{(k)}(x)$ are defined by
\begin{equation*}\label{Ankx-def01}
\sum_{n=0}^\infty A_n^{(k)}(x)\frac{z^n}{n!}=\left(\frac{1-x}{\mathrm{e}^{kz(x-1)}-x} \right)^{\frac{1}{k}},
\end{equation*}
where $k$ is a given positive integer. 
Savage-Viswanathan~\cite{Savage2012} discovered that   
\begin{equation}\label{Ankx-def02}
A_n^{(k)}(x)=\sum_{\pi\in\msn}x^{\exc(\pi)}k^{n-\cyc(\pi)}=E_n^{(1,k+1,2k+1,\ldots,(n-1)k+1)}(x),
\end{equation}
where $E_n^{(1,k+1,2k+1,\ldots,(n-1)k+1)}(x)$ is the {\it $\textbf{s}$-Eulerian polynomial} defined on 
$$\I_n^{(\textbf{s})}=\{(e_1,e_2,\ldots,e_n)\in\z^n:~~0\leqslant e_i<(i-1)k+1\}.$$
These polynomials satisfy the recursion
\begin{equation}\label{Anj-recu00}
A_{n+1}^{(k)}(x)=(nkx+1)A_n^{(k)}(x)+kx(1-x)\frac{\mathrm{d}}{\mathrm{d}x}A_n^{(k)}(x),~A_1^{(k)}(x)=1.
\end{equation}

For $\mathbf{m}=(m_1,m_2,\ldots,m_n)\in \mathbb{N}^n$, let $[\mathbf{n}]=\{1^{m_1},2^{m_2},\ldots,n^{m_n}\}$ be a multiset,
where the superscript denotes multiplicity. A {\it multipermutation} of $[\mathbf{n}]$ is a sequence of its elements. 
Given a multipermutation $\sigma$ of $[\mathbf{n}]$. Then $\sigma$ is called a {\it Stirling permutation} 
if $\sigma_i=\sigma_j$ and $i<s<j$, then $\sigma_s\geqslant\sigma_i$.
A {\it $k$-Stirling permutation} is a Stirling permutation over $[\textbf{n}]_k=\{1^{k},2^{k},\ldots,n^{k}\}$, where each entry appears $k$ times.
Let $\mq_n^{(k)}$ denotes the set of $k$-Stirling permutations over $[\textbf{n}]_k$. 
For $\sigma\in\mqn^{(k)}$, we always set $\sigma_0=\sigma_{kn+1}=0$ as usual, and we say that an value $\sigma_i$ is 
\begin{itemize}
   \item [$\bullet$] a {\it longest ascent-plateau} if $\sigma_{i-1}<\sigma_{i}=\sigma_{i+1}=\cdots=\sigma_{i+k-1}$, where $2\leqslant i\leqslant nk-k+1$;
   \item [$\bullet$] a {\it longest left ascent-plateau} if $\sigma_{i-1}<\sigma_{i}=\cdots=\sigma_{i+k-1}$, where $1\leqslant i\leqslant nk-k+1$;
   \item [$\bullet$] a {\it left-to-right minimum} if $\sigma_{i}<\sigma_j$ for $1\leqslant j<i\leqslant kn$ or $i=1$;
   \item [$\bullet$] a {\it right-to-left minimum} if $\sigma_{i}<\sigma_j$ for $1\leqslant i<j\leqslant kn$ or $i=kn$.
\end{itemize}

Let $\ap(\sigma)$ (resp.~$\lap(\sigma)$,~$\lrmin(\pi)$ and ~$\rlmin(\pi)$) be the number of longest ascent-plateaux (resp.~longest left ascent-plateaux, left-to-right minima and right-to-left minima) of $\sigma$. 
When $k=1$, the set $\mq_n^{(k)}$ reduces to $\msn$.
When $k=2$,
the set $\mq_n^{(k)}$ reduces to $\mqn$, which 
is the set of classical Stirling permutations introduced by Gessel-Stanley~\cite{Gessel78}. 
Except where explicitly stated, we assume that all Stirling permutations belong to $\mqn$. 
When $\sigma\in\mqn$, we use ascent-plateau (resp.~left ascent-plateau) for the longest ascent-plateau (resp.~longest left ascent-plateau).

There have been some recent work concerning the connections between permutations and Stirling permutations, 
see~\cite{Bona08,Brualdi20,Dzhumadil14,Haglund12,Kuba11,MaYeh17,Ma26} for details.
According to~\cite{Ma15}, we have  
\begin{equation}\label{Ankap}
A_n^{(k)}(x)=\sum_{\sigma\in\mqn^{(k)}}x^{\ap(\sigma)},~x^nA_n^{(k)}\left(\frac{1}{x}\right)=\sum_{\sigma\in\mqn^{(k)}}x^{\lap(\sigma)}.
\end{equation}

For the sake of convenience, in this paper, we always set 
$$M_n(x)=\sum_{\sigma\in\mqn}x^{\ap(\sigma)},~N_n(x)=\sum_{\sigma\in\mqn}x^{\lap(\sigma)}.$$
The first few of these polynomials are 
$M_1(x)=1,~M_2(x)=1+2x,~M_3(x)=1+10x+4x^2$, $N_1(x)=x,~N_2(x)=2x+x^2,~N_3(x)=4x+10x^2+x^3$.
From~\eqref{Ankap}, we see that $N_n(x)=x^nM_n(1/x)$. They appear frequently in combinatorics and geometry.
We list three of them as follows:
\begin{itemize}
   \item [$\bullet$] The polynomials $N_n(x)$ are the primitive Eulerian polynomials of type $B$, see~\cite{Bastidas23};
   \item [$\bullet$] The polynomials $N_n(x)$ are the enumerator of $n$-edge increasing ordered trees by number of bad vertices, see~\cite[Theorem~2]{Callan10};
   \item [$\bullet$] The polynomials $N_n(x)$ are the enumerator of perfect matchings of the set $[2n]$ by number of blocks with even maximal elements, see~\cite{MaYeh17}.
\end{itemize}

Let $\mbn$ denote the hyperoctahedral group of rank $n$. 
Elements of $\mbn$ are signed permutations of the set $\pm[n]=[n]\cup\{\overline{1},\ldots,\overline{n}\}$
with the property that $\sigma(\overline{i})=-\sigma(i)$ for all $i\in [n]$, where $\overline{i}=-i$.
Following Brenti~\cite{Brenti94},
the {\it type $B$ Eulerian polynomials} are defined by
$$B_n(x)=\sum_{\sigma\in\mbn}x^{\operatorname{des}_B(\sigma)},$$
where $\operatorname{des}_B(\pi)=\#\{i\in\{0,1,2,\ldots,n-1\}:~\pi(i)>\pi({i+1})~\&~\pi(0)=0\}$.
According to~\cite[Proposition~1]{MaYeh17}, we have
\begin{equation}\label{Convo01}
\begin{split}
2^nA_n(x)&=\sum_{i=0}^n\binom{n}{i}N_i(x)N_{n-i}(x),~B_n(x)=\sum_{i=0}^n\binom{n}{i}M_i(x)N_{n-i}(x).
\end{split}
\end{equation}
Here we collect the recursions of these polynomials:
\begin{equation*}
\begin{split}
A_{n+1}(x)&=(n+1)xA_{n}(x)+x(1-x)\frac{\mathrm{d}}{\mathrm{d}x}A_{n}(x),~A_0(x)=1;\\
B_{n+1}(x)&=(2nx+1+x)B_{n}(x)+2x(1-x)\frac{\mathrm{d}}{\mathrm{d}x}B_{n}(x),~B_0(x)=1;\\
M_{n+1}(x)&=(2nx+1)M_n(x)+2x(1-x)\frac{\mathrm{d}}{\mathrm{d}x}M_n(x),~M_0(x)=1;\\
N_{n+1}(x)&=(2n+1)xN_n(x)+2x(1-x)\frac{\mathrm{d}}{\mathrm{d}x}N_n(x),~N_0(x)=1.
\end{split}
\end{equation*}
Motivated by~\eqref{Ankx-def02} and~\eqref{Convo01}, there has been much work devoted to the $1/k$-Eulerian polynomials in recent years, 
see~\cite{Haglund12,Ji25,Liu21,Ma24,Savage15,Yan26} for instances.
For example, Liu~\cite{Liu21,Liu23} considered the $1/k$-Euler-Mahonian polynomials and established an extension of the
Haglund-Remmel-Wilson identity to $k$-Stirling permutations.  

This paper is organized as follows. In the next section, 
we provide a partial symmetric decomposition for the $1/k$-Eulerian polynomial. 
In Section~\ref{section4}, we find some unexpected applications of the $(p,q)$-Eulerian polynomials $A_n(x,y,p,q)=\sum_{\pi\in\msn}x^{\exc(\pi)}y^{\drop(\pi)}p^{\fix(\pi)}q^{\cyc(\pi)}$.
In particular, we introduce proper and improper ascent-plateaux of Stirling permutations.
Among other results, from Theorem~\ref{thm17}, we see that 
\begin{equation}\label{Ankk}
x^nA_n^{(k)}(1/x)=\sum_{\sigma\in\mqn^{(k)}}x^{\operatorname{lap}(\sigma)}=\sum_{\pi\in\msn}x^{\exc(\pi)+\fix(\pi)}k^{n-\cyc(\pi)},
\end{equation}
which gives a dual of the following result~~\cite{Ma15,Savage2012,Yan26}:
$$A_n^{(k)}(x)=\sum_{\sigma\in\mqn^{(k)}}x^{\operatorname{ap}(\sigma)}=\sum_{\pi\in\msn}x^{\exc(\pi)}k^{n-\cyc(\pi)}.$$
By introducing four new permutation statistics, in Theorem~\ref{thm24}, we provide a novel combinatorial interpretation of the 
following polynomial
$$A_n\left(x_1,y_1,\frac{p x_2+q y_2}{p+q},p+q\right),$$
which contains a great deal of information about permutations and Stirling permutations.
Some special cases and variations of this polynomial have been investigated by Chen-Fu~\cite{Chen25}, 
Ji-Lin~\cite{Ji24}, Ji~\cite{Ji25}, Savage-Viswanathan~\cite{Savage2012}, Xu-Zeng~\cite{Xu25} and Yan-Yang-Lin~\cite{Yan26}.
Moreover, a special case of Theorem~\ref{thm24} says that
$$\sum_{\pi\in\ms_{n+1}}x^{\operatorname{pdes}(\pi)}y^{\operatorname{pasc}(\pi)}
=\sum_{\pi\in\msn}\left(x+y\right)^{\fix(\pi)}2^{\cyc(\pi)-\fix(\pi)},$$
where $\operatorname{pdes}$ and $\operatorname{pasc}$ are proper descent and proper ascent statistics.
When $y=-x$, we get 
$$\sum_{\pi\in\ms_{n+1}}x^{\operatorname{pdes}(\pi)+\operatorname{pasc}(\pi)}(-1)^{\operatorname{pasc}(\pi)}
=\sum_{\pi\in\md_n}2^{\cyc(\pi)},$$
where $\md_n$ is the set of all derangements in $\msn$, i.e., $\md_n=\{\pi\in\msn: \fix(\pi)=0\}$. Moreover,
In Section~\ref{section3}, we first provide a symmetric decomposition for the enumerator 
of the joint distribution of ascent-plateaux and left ascent-plateaux, 
and then we show that the $q$-ascent-plateau polynomials are bi-$\gamma$-positive, where $q$ marks the number of left-to-right minima. 
\section{A partial symmetric decomposition of the $1/k$-Eulerian polynomial}
\subsection{Definitions and preliminaries}
\hspace*{\parindent}

Let $f(x)=\sum_{i=0}^nf_ix^i$ be a polynomial of degree $n$. We say that $f(x)$ is {\it unimodal} if there exists an index $k$ such that $f_0\leqslant f_1\leqslant f_2\leqslant \cdots \leqslant  f_{k-1}\leqslant f_k\geqslant f_{k+1}\geqslant \cdots\geqslant f_n$.
If $f(x)$ is symmetric, i.e., $f_i=f_{n-i}$ for all indices $0\leqslant i\leqslant n$,
then it can be uniquely expanded as
$$f(x)=\sum_{k=0}^{\lrf{{n}/{2}}}\gamma_kx^k(1+x)^{n-2k}.$$
The polynomial $f(x)$ is called {\it $\gamma$-positive} if $\gamma_k\geqslant 0$ for all $0\leqslant k\leqslant \lrf{{n}/{2}}$.
Clearly, $\gamma$-positivity implies unimodality. In~\cite{Foata70}, Foata and Sch\"utzenberger discovered that 
\begin{equation}\label{Anx-gamma}
A_n(x)=\sum_{k=1}^{\lrf{({n+1})/{2}}}a(n,k)x^k(1+x)^{n+1-2k},
\end{equation}
where the numbers $a(n,k)$ satisfy the recurrence relation
\begin{equation*}\label{ank-recu}
a(n,k)=ka(n-1,k)+(2n-4k+4)a(n-1,k-1),
\end{equation*}
with $a(1,1)=1$ and $a(1,k)=0$ for $k\geqslant 2$ (see~\cite[A101280]{Sloane}).
An elegant combinatorial proof of the $\gamma$-positivity of $A_n(x)$ can be found in~\cite{Branden08}.
The theory of $\gamma$-positivity has developed into 
an important and active topic in combinatorics and geometry, see~\cite{Athanasiadis17,Petersen15} for surveys.

The polynomial $f(x)$ is said to be {\it alternatingly increasing} if
$$f_0\leqslant f_n\leqslant f_1\leqslant f_{n-1}\leqslant\cdots \leqslant f_{\lrf{{(n+1)}/{2}}}.$$
Setting
\begin{equation*}\label{ax-bx-prop01}
a(x)=\frac{f(x)-x^{n+1}f(1/x)}{1-x},~b(x)=\frac{x^nf(1/x)-f(x)}{1-x},
\end{equation*} 
one can get that $f(x)= a(x)+xb(x)$, see~\cite{Beck2010,Schepers13}. 
The ordered pair of polynomials $(a(x),b(x))$ is called the {\it symmetric decomposition} of $f(x)$, since $a(x)$ and $b(x)$ are both symmetric.
\begin{lemma}[{\cite[Lemma~2.1]{Beck2019}}]\label{lemma-alt}
Let $(a(x),b(x))$ be the symmetric decomposition of $f(x)$, where $\deg f(x)=\deg a(x)=n$ and $\deg b(x)=n-1$.
Then $f(x)$ is alternatingly increasing if and only if both $a(x)$ and $b(x)$ are unimodal.
\end{lemma}

If $a(x)$ and $b(x)$ are both $\gamma$-positive, we say that $f(x)$ is {\it bi-$\gamma$-positive}.
It follows from Lemma~\ref{lemma-alt} that bi-$\gamma$-positivity is stronger than alternatingly increasing property. 
For instance,
\begin{align*}
M_2(x)&=1+x+x,~
M_3(x)=1+7x+x^2+x(3+3x),\\
M_4(x)&=1+29x+29x^2+x^3+x(7+31x+7x^2),\\
M_5(x)&=1+101x+321x^2+101x^3+x^4+x(15+195x+195x^2+15x^3).
\end{align*}
A fundamental result of~\cite{Ma24} can be restated as follows.
\begin{theorem}[{\cite[Theorem~3.6]{Ma24}}]\label{ANK-bigamma}
The $1/k$-Eulerian polynomials $A_n^{(k)}(x)$ are bi-$\gamma$-positive.
\end{theorem}
Recently, Yan-Yang-Lin~\cite{Yan26} provided a remarkable combinatorial interpretation 
for the bi-$\gamma$-coefficients of $A_n^{(k)}(x)$ by using the model of certain 
ordered labeled forests. In this paper, we give some original and substantial extensions of the $1/k$-Eulerian polynomials.

A {\it context-free grammar} $G$ over an alphabet
$V$ is defined as a set of substitution rules replacing a letter in $V$ by a formal function over $V$.
The formal derivative $D_G$ with respect to $G$ satisfies the derivation rules:
$D_G(u+v)=D_G(u)+D_G(v),~D_G(uv)=D_G(u)v+uD_G(v)$.
So the {\it Leibniz rule} holds:
\begin{equation*}\label{Leibniz}
D_G^n(uv)=\sum_{k=0}^n\binom{n}{k}D_G^k(u)D_G^{n-k}(v).
\end{equation*}

There have been two methods proposed in recent years: grammatical labelings~\cite{Chen17,Dumont96,Ma26} and the change of grammars~\cite{Chen23,Ji25,Ma19,Ma24}. 
In~\cite{Dumont96}, Dumont discovered the context-free grammar for Eulerian polynomials by using a grammatical labeling of circular permutations.
\begin{proposition}[{\cite[Section~2.1]{Dumont96}}]\label{grammar03}
Let $G=\{a\rightarrow ab, b\rightarrow ab\}$.
Then for $n\geqslant 1$, one has
\begin{equation*}
D_{G}^n(a)=D_{G}^n(b)=b^{n+1}A_n\left(\frac{a}{b}\right).
\end{equation*}
\end{proposition}
\subsection{Partial symmetric decompositions}
\hspace*{\parindent}

In the sequel, we shall first deduce a grammar for the $1/k$-Eulerian polynomials by using differential operators.
According to~\cite[Eq.~(5)]{Savage2012}, we have
\begin{equation}\label{Ankx-def04}
\sum_{t\geqslant 0}\binom{t-1+\frac{1}{k}}{t}(kt+1)^nx^t=\frac{A_n^{(k)}(x)}{(1-x)^{n+\frac{1}{k}}}.
\end{equation}
Since $$\frac{1}{(1-x)^{1/k}}=\sum_{t\geqslant 0}\binom{t-1+\frac{1}{k}}{t}x^t,$$
it follows from~\eqref{Ankx-def04} that
\begin{equation}\label{Ankx-def05}
\left(kx\frac{\mathrm{d}}{\mathrm{d}x}+1\right)^n\frac{1}{(1-x)^{1/k}}=\frac{A_n^{(k)}(x)}{(1-x)^{n+\frac{1}{k}}}.
\end{equation}
Using~\eqref{Anj-recu00}, it is routine to verify that an equivalent of~\eqref{Ankx-def05} is given as follows:
\begin{equation}\label{Ankx-def06}
\left(kx\frac{\mathrm{d}}{\mathrm{d}x}\right)^n\frac{1}{(1-x)^{1/k}}=\frac{x^nA_n^{(k)}(1/x)}{(1-x)^{n+\frac{1}{k}}}.
\end{equation}

Setting $$T=kx\frac{\mathrm{d}}{\mathrm{d}x},~~a=\frac{1}{(1-x)^{1/k}},~~b=\left(\frac{x}{1-x}\right)^{1/k},$$
then $T(a)=ab^k,~T(b)=a^kb$. So we get the grammar $G=\{a\rightarrow ab^k,~b\rightarrow a^kb\}$.
Note that $ab^{kn}={x^n}/{(1-x)^{n+\frac{1}{k}}}$ and ${a^k}/{b^k}={1}/{x}$.
By~\eqref{Ankx-def06}, we arrive at the following result. 
\begin{lemma}\label{Ank-grammar}
Let $G=\{a\rightarrow ab^k,~b\rightarrow a^kb\}$. We have $D_{G}^n(a)=ab^{kn}A_n^{(k)}\left(\frac{a^k}{b^k}\right)$.
\end{lemma}

As an extension of~\eqref{Anx-gamma},
we can now conclude the first main result of this paper.
\begin{theorem}\label{thm1}
The $1/k$-Eulerian polynomials are partial $\gamma$-positive. More precisely, 
we have
\begin{equation}\label{thm001}
A_{n+1}^{(k)}(x)=\sum_{i\geqslant 0}(1+kx)^i\sum_{j\geqslant 0}\gamma_{n,i,j}(k)x^j(1+x)^{n-i-2j}.
\end{equation}
Let $\gamma_{n}(x,y)=\sum_{i\geqslant 0}\sum_{j\geqslant 0}\gamma_{n,i,j}(k)x^iy^j$. Then the coefficient polynomials $\gamma_{n}(x,y)$ satisfy
$$\gamma_{n+1}(x,y)=(x+2kny)\gamma_{n}(x,y)+y(k+k^2-2kx)\frac{\partial}{\partial x}\gamma_{n}(x,y)+ky(1-4y)\frac{\partial}{\partial y}\gamma_{n}(x,y),$$
with the initial values $\gamma_{1}(x,y)=x$ and $\gamma_{2}(x,y)=x^2+(k+k^2)y$.
\end{theorem}
\begin{proof}
Consider $G=\{a\rightarrow ab^k,~b\rightarrow a^kb\}$. Setting $\alpha=a^k$ and $\beta=b^k$, then 
$D_G(\alpha)=D_G(\beta)=k\alpha\beta$. So we get a new grammar $H=\{a\rightarrow a\beta,~b\rightarrow \alpha b,~\alpha\rightarrow k\alpha\beta,~\beta\rightarrow k\alpha\beta\}$.
It follows from~\eqref{Ankx-def02} and Lemma~\eqref{Ank-grammar} that 
\begin{equation}\label{DHA}
D_H^n(a)|_{\beta=1,\alpha=x}=aA_n^{(k)}(x)=a\sum_{\pi\in\ms_{n+1}}x^{\exc(\pi)}k^{n-\cyc(\pi)}.
\end{equation}

Consider a change of $H$. 
Set $I=a\beta,~w=\beta+k\alpha,~u=\alpha\beta$ and $v=\alpha+\beta$. Then
$$D_H(I)=Iw,~D_H(w)=(k+k^2)u,~D_H(u)=kuv,~D_H(v)=2ku.$$
Note that $D_H^2(I)=Iw^2+(k+k^2)Iu,~D_H^3(I)=Iw^3+3(k+k^2)Iwu+k(k+k^2)Iuv$.
By induction, it is routine to check that there exists nonnegative integers $\gamma_{n,i,j}(k)$ such that 
\begin{equation}\label{DHI}
D_H^n(I)=I\sum_{i\geqslant 0}w^i\sum_{j\geqslant 0}\gamma_{n,i,j}(k)u^jv^{n-i-2j},
\end{equation}
where the coefficients $\gamma_{n,i,j}(k)$ satisfy the recursion
$$\gamma_{n+1,i,j}(k)=\gamma_{n,i-1,j}(k)+(k+k^2)(i+1)\gamma_{n,i+1,j-1}(k)+kj\gamma_{n,i,j}(k)+2k(n-i-2j+2)\gamma_{n,i,j-1}(k),$$
with $\gamma_{1,1,0}(k)=1$ and $\gamma_{1,i,j}(k)=0$ for all $(i,j)\neq (1,0)$. Multiplying both sides of the above recursion by $x^iy^j$ and summing over all $i,j$, we 
get the recursion of $\gamma_n(x,y)$.
By~\eqref{DHA}, we obtain $$aA_{n+1}^{(k)}(x)=D_H^{n+1}(a)|_{\beta=1,\alpha=x}=D_H^{n}(I)|_{\beta=1,\alpha=x}.$$
Then in~\eqref{DHI}, by taking $I=a,~w=1+kx,~u=x$ and $v=1+x$, we get the desired result.
\end{proof}

For example, we have $A_2^{(k)}(x)=1+kx,~A_3^{(k)}(x)=(1+kx)^2+(k+k^2)x$ and
\begin{align*}
A_4^{(k)}(x)&=(1 + kx)^3 +3(k + k^2) (1 + kx)  x + k (k + k^2) x(1 + x).
\end{align*}
When $k=2$, we have
$$M_{n+1}(x)=\sum_{\sigma\in\mq_{n+1}}x^{\ap(\sigma)}
=\sum_{i\geqslant 0}(1+2x)^i\sum_{j\geqslant 0}\zeta_{n,i,j}(2x)^j(1+x)^{n-i-2j},$$
where $2^j\zeta_{n,i,j}=\gamma_{n,i,j}(2)$ and $\zeta_{n,i,j}$ satisfy the recursion
$$\zeta_{n+1,i,j}=\zeta_{n,i-1,j}+3(i+1)\zeta_{n,i+1,j-1}+2j\zeta_{n,i,j}+2(n-i-2j+2)\zeta_{n,i,j-1},$$
with $\zeta_{1,1,0}=1$ and $\zeta_{1,i,j}=0$ for all $(i,j)\neq (1,0)$.
\section{Unexpected applications of the $\fix$ and $\cyc$ $(p,q)$-Eulerian polynomials}\label{section4}
The $\fix$ and $\cyc$ {\it $(p,q)$-Eulerian polynomials} $A_n(x,y,p,q)$ are defined by
\begin{equation*}
A_n(x,y,p,q)=\sum_{\pi\in\msn}x^{\exc(\pi)}y^{\drop(\pi)}p^{\fix(\pi)}q^{\cyc(\pi)}.
\end{equation*}
This $(p,q)$-Eulerian polynomial contains a great deal of information about permutations and colored permutations, see~\cite{Ji25,Ma24,Zeng12,Yan26} for instance.
Using the {\it exponential formula}, Ksavrelof and Zeng~\cite{Zeng02} found that
\begin{equation*}\label{Anxpq-Zeng}
\sum_{n=0}^\infty A_n(x,1,p,q)\frac{z^n}{n!}=\left(\frac{(1-x)\mathrm{e}^{pz}}{\mathrm{e}^{xz}-x\mathrm{e}^{z}}\right)^q.
\end{equation*}
Below are the first few of these polynomials:
\begin{align*}
A_1(x,1,p,q)&=pq,~
A_2(x,1,p,q)=p^2q^2+qx,~
A_3(x,1,p,q)=p^3q^3+(q+3pq^2)x+qx^2,\\
A_4(x,1,p,q)&=p^4q^4+(q+4pq^2+6p^2q^3)x+(4q+3q^2+4pq^2)x^2+qx^3.
\end{align*}
Note that $\exc(\pi)+\drop(\pi)+\fix(\pi)=n$ for $\pi\in\msn$. So we have 
\begin{equation}\label{Anxpq-EGF}
\sum_{n\geqslant 0}A_n(x,y,p,q)\frac{z^n}{n!}=\sum_{n\geqslant 0}y^nA_n\left(\frac{x}{y},1,\frac{p}{y},q\right)\frac{z^n}{n!}=\left(\frac{(y-x)\mathrm{e}^{pz}}{y\mathrm{e}^{xz}-x\mathrm{e}^{yz}}\right)^q.
\end{equation}

The following result will be used repeatedly in our discussion.
\begin{lemma}[{\cite[Lemma~3.12]{Ma24}}]\label{lemmacycle}
If
\begin{equation}\label{G1}
G=\{I\rightarrow Ipq, p\rightarrow xy, x\rightarrow xy, y\rightarrow xy\},
\end{equation}
then we have
$$D_{G}^n(I)=I\sum_{\pi\in\msn}x^{\exc(\pi)}y^{\drop(\pi)}p^{\fix(\pi)}q^{\cyc(\pi)}.$$
\end{lemma}

\subsection{A remark on $(\alpha,\beta)$-Eulerian polynomials}
\hspace*{\parindent}

Let $\pi\in\msn$. Recall that $$\asc(\pi)=\#\{i\in[n]: ~\pi(i-1)<\pi(i)~\&~\pi(0)=0\},$$
$$\des(\pi)=\#\{i\in[n]: ~\pi(i)>\pi(i+1)~\&~\pi(n+1)=0\}.$$
We define 
 $$\asc^*(\pi)=\#\{i\in [n-1]: ~\pi(i)<\pi(i+1)\},$$
 $$\des^*(\pi)=\#\{i\in[n-1]: ~\pi(i)>\pi(i+1)\}.$$
Clearly, $\asc(\pi)=\asc^*(\pi)+1$ and $\des(\pi)=\des^*(\pi)+1$.
A {\it left-to-right maximum} of $\pi$ is an entry $\pi(i)$ such that $\pi(i)>\pi(j)$ for every $j<i$ and 
a {\it right-to-left maximum} of $\pi$ is an index $i$ such that $\pi(i)>\pi(j)$ for every $j>i$.
Let $\operatorname{lrmax}(\pi)$ and $\operatorname{rlmax}(\pi)$ denote the number of left-to-right maxima and right-to-left maxima of $\pi$, respectively.
The {\it $(\alpha,\beta)$-Eulerian polynomials} are defined by
$$A_n(x,y|\alpha,\beta)=\sum_{\pi\in\ms_{n+1}}x^{\asc^*(\pi)}y^{\des^*(\pi)}\alpha^{\operatorname{lrmax}(\pi)-1}\beta^{\operatorname{rlmax}(\pi)-1}.$$
In particular, $A_1(x,y|\alpha,\beta)=\alpha x+\beta y$ and $A_2(x,y|\alpha,\beta)=(\alpha x+\beta y)^2+(\alpha+\beta)xy$.
Carlitz-Scoville~\cite{Carlitz74} introduced the $(\alpha,\beta)$-Eulerian polynomials and obtained that 
\begin{equation}\label{Carlitz}
\sum_{n\geqslant 0}A_n(x,y|\alpha,\beta)\frac{z^n}{n!}=\left(1+xF(x,y;z)\right)^\alpha\left(1+yF(x,y;z)\right)^\beta,
\end{equation}
where by convention $A_0(x,y|\alpha,\beta)=1$ and $$F(x,y;z)=\frac{\mathrm{e}^{xz}-\mathrm{e}^{yz}}{x\mathrm{e}^{yz}-y\mathrm{e}^{xz}}.$$
By taking the complements of permutations, Ji~\cite{Ji25} found that 
$$A_n(x,y|\alpha,\beta)=\sum_{\pi\in\ms_{n+1}}x^{\des^*(\pi)}y^{\asc^*(\pi)}\alpha^{\operatorname{lrmin}(\pi)-1}\beta^{\operatorname{rlmin}(\pi)-1}.$$
A grammatical interpretation of the $(\alpha,\beta)$-Eulerian polynomials is given as follows.
\begin{lemma}[{\cite[Theorem~3.1]{Ji25}}]\label{lemmaJi}
If $H=\{a\rightarrow a\alpha x, b\rightarrow b\beta y, x\rightarrow xy, y\rightarrow xy\}$,
then we have $D_{H}^n(ab)=abA_n(x,y|\alpha,\beta)$.
\end{lemma}

Recently, various bijections between permutations and increasing trees are repeatedly discovered, 
see~\cite{Chen23,Chen25,Yan26} for instances. When $\alpha=\beta$, the $(\alpha,\beta)$-Eulerian polynomials are reduced to the $\alpha$-Eulerian polynomials.
Using a labeling scheme for increasing binary trees, 
Chen-Fu~\cite{Chen25} obtained a combinatorial interpretation of the $\gamma$-coefficients of the $\alpha$-Eulerian polynomials
in terms of forests of planted 0-1-2-plane trees. 
We can now present the following result, which has also been discussed by Xu-Zeng~\cite[Eq.~(1.9)]{Xu25}.  
\begin{proposition}\label{thmab}
The $(\alpha,\beta)$-Eulerian polynomials can be obtained from the $\fix$ and $\cyc$
$(p,q)$-Eulerian polynomials by special parametrizations:
$$A_n(x,y|\alpha,\beta)=A_n\left(x,y,\frac{\alpha x+\beta y}{\alpha+\beta},\alpha+\beta\right).$$
\end{proposition}
\begin{proof}
For the grammar $H=\{a\rightarrow a\alpha x, b\rightarrow b\beta y, x\rightarrow xy, y\rightarrow xy\}$, we see that 
$$D_H(ab)=ab(\alpha x+\beta y),~D_H(\alpha x+\beta y)=(\alpha+\beta)xy.$$
Setting $I=ab$ and $J=\alpha x+\beta y$, then we get the following grammar:
$$F=\{I\rightarrow IJ,~J\rightarrow (\alpha+\beta)xy,~ x\rightarrow xy,~y\rightarrow xy\}.$$
As a variant of Lemma~\ref{lemmacycle}, we claim that 
\begin{equation}\label{DFI}
D_F^n(I)=I\sum_{\pi\in\msn}x^{\exc(\pi)}y^{\drop(\pi)}J^{\fix(\pi)}(\alpha+\beta)^{\cyc(\pi)-\fix(\pi)},
\end{equation}
which can be easily verified by using the following grammatical labeling of $\pi$:
\begin{itemize}
  \item [\rm ($L_1$)]If $\pi(i)$ is an excedance, then put a superscript label $x$ just before $\pi(i)$;
  \item [\rm ($L_2$)]If $\pi(i)$ is a drop, then put a superscript label $y$ just right after $i$;
 \item [\rm ($L_3$)]If $\pi(i)$ is a fixed point, then put a superscript label $J$ right after $i$;
 \item [\rm ($L_4$)]Put a superscript label $I$ right after $\pi$ and put a subscript label $\alpha+\beta$ right after each cycle with at least two elements.
\end{itemize}

When $\pi\in\msn$ is represented as the composition of disjoint cycles, we always write $\pi$
by using its \emph{standard cycle decomposition}, in which each
cycle is written with its smallest entry first and the cycles are written in ascending order of their
first entry.
For example, below shows the grammatical labeling of $\pi=(1,3,5)(2,6,4)(7)$:
$$(1^x3^x5^y)_{\alpha+\beta}(2^x6^y4^y)_{\alpha+\beta}(7^J)^I.$$

By taking $I=ab$ and $J=\alpha x+\beta y$ in~\eqref{DFI}, it follows from Lemma~\ref{lemmaJi} that
$$A_n(x,y|\alpha,\beta)=\sum_{\pi\in\msn}x^{\exc(\pi)}y^{\drop(\pi)}(\alpha x+\beta y)^{\fix(\pi)}(\alpha+\beta)^{\cyc(\pi)-\fix(\pi)},$$
as desired. This completes the proof.
\end{proof}

Combining~\eqref{Anxpq-EGF} and Theorem~\ref{thmab}, we immediately get a simple variant of~\eqref{Carlitz}.
\begin{corollary}
We have 
$$\sum_{n\geqslant 0}A_n(x,y|\alpha,\beta)\frac{z^n}{n!}
=\left(\frac{(y-x)\mathrm{e}^{\frac{\alpha x+\beta y}{\alpha+\beta}z}}{y\mathrm{e}^{xz}-x\mathrm{e}^{yz}}\right)^{\alpha+\beta}.$$
\end{corollary}

Given a permutation $\pi\in\msn$ and set $\pi(0)=\pi(n+1)=0$. We say that $\pi(i)$ is
\begin{itemize}
  \item a {\it peak} if $\pi({i-1})<\pi(i)>\pi({i+1})$;
  \item a {\it valley} if $\pi({i-1})>\pi(i)<\pi({i+1})$;
  \item a {\it double ascent} if $\pi({i-1})<\pi(i)<\pi({i+1})$;
  \item a {\it double descent} if $\pi({i-1})>\pi(i)>\pi({i+1})$.
\end{itemize} 
Let $\operatorname{pk}(\pi)$ (resp.~$\operatorname{val}(\pi)$,~$\operatorname{dasc}(\pi)$,~$\operatorname{ddes}(\pi)$) be the number of peaks (resp.~valleys, double ascents, double descents) of $\pi$.
It should be noted that $\operatorname{pk}(\pi)=\operatorname{val}(\pi)+1$, $\asc(\pi)=\operatorname{dasc}(\pi)+\operatorname{pk}(\pi)$ and 
$\des(\pi)=\operatorname{ddes}(\pi)+\operatorname{pk}(\pi)$.
Ji~\cite{Ji25} provided an exponential formula for the following six-variable polynomials
$$A_n(u_1,u_2,u_3,u_4|\alpha,\beta)=\sum_{\pi\in\ms_{n+1}}(u_1u_2)^{\operatorname{val}(\pi)}
u_3^{\operatorname{dasc}(\pi)}u_4^{\operatorname{ddes}(\pi)}\alpha^{\operatorname{lrmax}(\pi)-1}\beta^{\operatorname{rlmax}(\pi)-1}.$$
In~\cite[Eq.~(1.10)]{Xu25}, Xu-Zeng found that 
\begin{equation}\label{Xu01}
A_n(u_1,u_2,u_3,u_4|\alpha,\beta)=A_n\left(x,y,\frac{p u_3+q u_4}{p+q},p+q\right),
\end{equation}
where $xy=u_1u_2$ and $x+y=u_3+u_4$. In Theorem~\ref{thm24}, we shall give a different formula. 
\subsection{Proper and improper left ascent-plateaux of Stirling permutations}
\hspace*{\parindent}

\begin{definition}
Given $\sigma\in\mqn$. For $1\leqslant i\leqslant 2n-1$, we say that an entry $\sigma_i$ is a {\it proper left ascent-plateau} if it
satisfies two conditions: $(i)$ $\sigma_i$ is a left ascent-plateau, i.e., $\sigma_{i-1}<\sigma_i=\sigma_{i+1}$; $(ii)$ if $\sigma_j<\sigma_i$, then $j<i$, or $i=1$ and $\sigma_1=\sigma_2=1$.
In other words, all entries less than $\sigma_i$ are on the left of $\sigma_i$.
\end{definition}
Let $\operatorname{plap}(\sigma)$ be the number of {\it proper left ascent-plateaux} of $\sigma$, and the number of {\it improper left ascent-plateaux} is given by
$\operatorname{implap}(\sigma)=\lap(\sigma)-\operatorname{plap}(\sigma)$.
\begin{example}
In $\mq_2$, we have $$\operatorname{plap}(\textbf{1}1\textbf{2}2)=2,~\operatorname{plap}(1221)=0,~\operatorname{plap}(2211)=0,$$
 $$\operatorname{implap}(1122)=0,~\operatorname{implap}(1\textbf{2}21)=1,~\operatorname{implap}(\textbf{2}211)=1.$$
\end{example}

\begin{example}
If $\sigma=11245547723366\in\mq_7$, then $1$ and $3$ are both proper left ascent-plateaux, while $5,6$ and $7$ are all improper left ascent-plateaux.
\end{example}

\begin{lemma}\label{lemmaap}
Let $G=\{I\rightarrow Iy,~y\rightarrow 2xz,~x\rightarrow 2xz,~z\rightarrow 2xz\}$. Then we have 
$$D_G^n(I)=I\sum_{\sigma\in\mqn}x^{\operatorname{implap}(\sigma)}y^{\operatorname{plap}(\sigma)}z^{n-\operatorname{lap}(\sigma)}.$$
\end{lemma}
\begin{proof}
Given $\sigma\in\mqn$. 
First put a label $I$ at the end of $\sigma$, and we label $\sigma_i$ as follows:
\begin{itemize}
  \item [$(i)$] If $\sigma_i$ is a proper left ascent-plateau, then we label the two positions just before and right after $\sigma_i$ by a superscript label $y$, i.e., $\sigma_{i-1}\overbrace{\sigma_{i}}^y\sigma_{i+1}$;
  \item [$(ii)$] If $\sigma_i$ is an improper left ascent-plateau, then we label the two positions just before and right after $\sigma_i$ by a superscript label $x$, i.e., $\sigma_{i-1}\overbrace{\sigma_{i}}^x\sigma_{i+1}$;
  \item [$(iii)$] Except the above cases, there still have an even number of positions (maybe empty) and so we attach a subscript label $z$ to record pairwise nearest elements from left to right.
\end{itemize}
When $n=1,2$, the labeled elements are listed as follows:
$$\overbrace{1}^y1^I,~\overbrace{1}^y1\overbrace{2}^y2^I,~\underbrace{1\overbrace{2}^x2}_z1^I,~\overbrace{2}^x2\underbrace{1}_z1^I.$$
Note that $D_G(I)=Iy$ and $D_G^2(I)=Iy^2+2Ixz$. Hence the result holds for $n=1,2$. 
We proceed by induction. Now assume that $m\geqslant 2$ and $\sigma\in\mq_{m-1}$. Suppose that 
$\sigma'$ is Stirling permutation in $\mq_m$ created
from $\sigma$ by inserting the string $mm$.
Then the changes of labeling can be illustrated as follows:
$$\sigma^I\mapsto \sigma \overbrace{m}^ym^I;$$
$$\cdots\overbrace{\sigma_{i}}^x\sigma_{i+1}\cdots\mapsto \cdots \overbrace{m}^xm\underbrace{{\sigma_{i}}}_z\sigma_{i+1}\cdots~{\text{or}}~\cdots\underbrace{{\sigma_{i}}\overbrace{m}^xm}_z\sigma_{i+1}\cdots;$$
$$\cdots\overbrace{\sigma_{i}}^y\sigma_{i+1}\cdots\mapsto \cdots \overbrace{m}^xm\underbrace{{\sigma_{i}}}_z\sigma_{i+1}\cdots~{\text{or}}~\cdots\underbrace{{\sigma_{i}}\overbrace{m}^xm}_z\sigma_{i+1}\cdots.$$
Clearly, if we insert the string $mm$ into one of the two positions 
with a label $z$, we always get the same changes of labels as the insertion of $mm$ into one of the two position with a label $x$. 
Summing up all the cases shows that the assertion is valid for $m$ and the insertion of the string $mm$ corresponds to the operator $D_G$.
This completes the proof.
\end{proof}

\begin{theorem}\label{thmproper}
We have
$$\sum_{\sigma\in\mqn}x^{\operatorname{implap}(\sigma)}y^{\operatorname{plap}(\sigma)}=\sum_{\pi\in\msn}x^{\exc(\pi)}y^{\fix(\pi)}2^{n-\cyc(\pi)}.$$
When $y=x$, we get 
$$\sum_{\sigma\in\mqn}x^{\operatorname{lap}(\sigma)}=\sum_{\pi\in\msn}x^{\exc(\pi)+\fix(\pi)}2^{n-\cyc(\pi)}.$$
\end{theorem}
\begin{proof}
We first introduce a grammatical labeling of $\pi\in\msn$:
\begin{itemize}
  \item [\rm ($L_1$)]If $\pi(i)$ is an excedance, then put a superscript label $x$ just before $\pi(i)$;
  \item [\rm ($L_2$)]If $\pi(i)$ is a drop, then put a superscript label $y$ just right after $i$;
 \item [\rm ($L_3$)]If $\pi(i)$ is a fixed point, then put a superscript label $p$ right after $i$;
 \item [\rm ($L_4$)]Put a superscript label $I$ at the end of $\pi$;
  \item [\rm ($L_5$)]For any cycle with at least two elements, put a subscript label $q$ right after each element except the smallest one.
\end{itemize}
For example, below shows the grammatical labeling of $\pi=(1,3,5)(2,6,4)(7)$:
$$(1^x3^x_q5^y_q)(2^x6^y_q4^y_q)(7^p)^I.$$
As a variant of Lemma~\ref{lemmacycle}, it is easy to verify that 
If
\begin{equation}\label{keylemma}
G=\{I\rightarrow Ip, p\rightarrow qxy, x\rightarrow qxy, y\rightarrow qxy\},
\end{equation}
then we have
$$D_{G}^n(I)=I\sum_{\pi\in\msn}x^{\exc(\pi)}y^{\drop(\pi)}p^{\fix(\pi)}q^{n-\cyc(\pi)}.$$
Comparing this Lemma~\eqref{lemmaap}, we get the desired result.
\end{proof}

In the same way, for $\sigma\in\mq_n^{(k)}$, we say that an entry $\sigma_i$ is a {\it proper left ascent-plateau} if it
satisfies two conditions: $(i)$ $\sigma_i$ is a longest left ascent-plateau, i.e., $\sigma_{i-1}<\sigma_{i}=\cdots=\sigma_{i+k-1}$; $(ii)$ if $\sigma_j<\sigma_i$, then $j<i$, or $i=1$ and $\sigma_1=\sigma_2=\cdots=\sigma_k=1$. Let $\operatorname{plap}(\sigma)$ be the number of proper left ascent-plateaux of $\sigma$, and 
set $\operatorname{implap}(\sigma)=\lap(\sigma)-\operatorname{plap}(\sigma)$.
The following result follows from the same argument
as the proof of Lemma~\ref{lemmaap}, and we omit its proof for simplicity.
\begin{lemma}\label{lemmaapp}
Let $G=\{I\rightarrow Iy,~y\rightarrow kxz,~x\rightarrow kxz,~z\rightarrow kxz\}$. Then we have 
$$D_G^n(I)=I\sum_{\sigma\in\mqn^{(k)}}x^{\operatorname{implap}(\sigma)}y^{\operatorname{plap}(\sigma)}z^{n-\operatorname{lap}(\sigma)}.$$
\end{lemma}
Comparing Lemma~\ref{lemmaapp} with~\eqref{keylemma}, we immediately get the following result.
\begin{theorem}\label{thm17}
We have
$$\sum_{\sigma\in\mqn^{(k)}}x^{\operatorname{implap}(\sigma)}y^{\operatorname{plap}(\sigma)}
=\sum_{\pi\in\msn}x^{\exc(\pi)}y^{\fix(\pi)}k^{n-\cyc(\pi)}.$$
\end{theorem}
The reader is referred to~\eqref{Ankk} for a special case of Theorem~\ref{thm17}, which said that
\begin{equation*}
x^nA_n^{(k)}(1/x)=\sum_{\sigma\in\mqn^{(k)}}x^{\operatorname{lap}(\sigma)}=\sum_{\pi\in\msn}x^{\exc(\pi)+\fix(\pi)}k^{n-\cyc(\pi)}.
\end{equation*}
\subsection{Proper ascents, improper ascents, proper descents and improper descents}
\hspace*{\parindent}

\begin{definition}
An entry $\pi(i)$ of $\pi\in\msn$ is called a {\it proper ascent} if it
satisfies two conditions: $(i)$ $\pi(i-1)<\pi(i)$, 
where $2\leqslant i\leqslant n$; $(ii)$ if $\pi(j)<\pi(i)$, then $j<i$.
\end{definition}
\begin{definition}
An entry $\pi(i)$ of $\pi\in\msn$ is called a {\it proper descent} if it
satisfies two conditions: $(i)$ $\pi(i)>\pi(i+1)$, 
where $i\in[n-1]$; $(ii)$ if $\pi(j)<\pi(i)$, then $j>i$.
\end{definition}
Let $\operatorname{pasc}(\pi)$ (resp.~$\operatorname{pdes}(\pi)$) be the 
number of proper ascents (resp.~proper descents) of $\pi$. 
The number of {\it improper ascents} and {\it improper descents} of $\pi$ are respectively defined by
$$\operatorname{impasc}(\pi)=\asc^*(\pi)-\operatorname{pasc}(\pi),$$
$$\operatorname{impdes}(\pi)=\des^*(\pi)-\operatorname{pdes}(\pi).$$

\begin{example} 
We have
$$\operatorname{{pasc}}(1\textbf{23})=2,~\operatorname{{pasc}}({1}32)=0,~\operatorname{{pasc}}(21\textbf{3})=1,$$
$$\operatorname{{pasc}}(231)=0,~\operatorname{{pasc}}(31\textbf{2})=1,~\operatorname{{pasc}}(321)=0;$$
$$\operatorname{{pdes}}({123})=0,~\operatorname{{pdes}}({1}32)=0,~\operatorname{{pdes}}(\textbf{2}1{3})=1,$$
$$\operatorname{{pdes}}(231)=0,~\operatorname{{pdes}}(\textbf{3}1{2})=1,~\operatorname{{pdes}}(\textbf{32}1)=2.$$
\end{example}

We now ready to give one of the central findings of this paper, which is different from~\eqref{Xu01}, since in~\eqref{Xu01},  
it must be $xy=u_1u_2$ and $x+y=u_3+u_4$.
\begin{theorem}\label{thm24}
We have
$$\sum_{\pi\in\ms_{n+1}}x_1^{\operatorname{impdes}(\pi)}x_2^{\operatorname{pdes}(\pi)}y_1^{\operatorname{impasc}(\pi)}
y_2^{\operatorname{pasc}(\pi)}p^{\lrmin(\pi)-1}q^{\rlmin(\pi)-1}
=A_n\left(x_1,y_1,\frac{p x_2+q y_2}{p+q},p+q\right).$$
\end{theorem}
\begin{proof}
We now define the {\it $A$-labeling} for $\pi\in \msn$. 
Put a label $I$ at the front of $\pi$ and put a label $J$ at the end of $\pi$.
Then the grammatical labeling of $\pi(i)$ can be given as follows:
\begin{itemize}
  \item [\rm ($L_1$)]If $\pi(i)$ is an improper descent, then put a label $x_1$ right after $\pi(i)$;
  \item [\rm ($L_2$)]If $\pi(i)$ is a proper descent, then put a label $x_2$ right after $\pi(i)$;
 \item [\rm ($L_3$)]If $\pi(i)$ is an improper ascent, then put a label $y_1$ just before $\pi(i)$;
 \item [\rm ($L_4$)]If $\pi(i)$ is an proper ascent, then put a label $y_2$ just before $\pi(i)$;
  \item [\rm ($L_5$)]If $\pi(i)$ is a left-to-right minimum and $\pi(i)\neq1$, then put a label $p$ right after $\pi(i)$;
  \item [\rm ($L_6$)]If $\pi(i)$ is a right-to-left right minimum and $\pi(i)\neq1$, then put a label $q$ right after $\pi(i)$.
\end{itemize}
The weight of $\pi$ is
defined as the product of its labels, that is,
$$w(\pi)=IJx_1^{\operatorname{impdes}(\pi)}x_2^{\operatorname{pdes}(\pi)}y_1^{\operatorname{impasc}(\pi)}y_2^{\operatorname{pasc}(\pi)}p^{\lrmin(\pi)-1}q^{\rlmin(\pi)-1}.$$
When $n=1,2,3$, the labeled permutations can listed as follows:
$^I1^J,~^I1^{y_2}2_q^J,~^I2^{x_2}_p1^J$,
$$~^I1^{y_2}2_q^J\rightarrow \left\{
  \begin{array}{ll}
    ^I1^{y_2}2^{y_2}_q3_q^J, &  \\
   ^I1^{y_1}3^{x_1}2_q^J, &\\
   ^I3_p^{x_2}12_q^J;
  \end{array}
\right.~~^I2^{x_2}_p1^J\rightarrow \left\{
  \begin{array}{ll}
    ^I2_p^{x_2}1^{y_2}3_q^J, &  \\
   ^I2_p^{y_1}3^{x_1}1^J, &\\
   ^I32_p^{x_2}1^J.
  \end{array}
\right.$$
By induction, it is routine to verify that if 
\begin{equation}\label{G3}
G=\{I\rightarrow Ipx_2,~J\rightarrow Jqy_2,~x_1\rightarrow x_1y_1,~x_2\rightarrow x_1y_1,~y_1\rightarrow x_1y_1,~y_2\rightarrow x_1y_1\},
\end{equation}
then we have $$D_G^n(IJ)=IJ\sum_{\pi\in\ms_{n+1}}x_1^{\operatorname{impdes}(\pi)}x_2^{\operatorname{pdes}(\pi)}
y_1^{\operatorname{impasc}(\pi)}y_2^{\operatorname{pasc}(\pi)}p^{\lrmin(\pi)-1}q^{\rlmin(\pi)-1}.$$
For the grammar given by~\eqref{G3}, it should be noted that
$$D_G(IJ)=IJ(px_2+qy_2),~D_G(px_2+qy_2)=(p+q)x_1y_1.$$
Setting $a=IJ$ and $b=px_2+qy_2$, then we get a new grammar
$$H=\{a\rightarrow ab,~b\rightarrow (p+q)x_1y_1,~x_1\rightarrow x_1y_1,~y_1\rightarrow x_1y_1\}.$$ 
As a variant of Lemma~\ref{lemmacycle}, it is easy to verify that 
$$D_H^n(a)=a\sum_{\pi\in\msn} x_1^{\exc(\pi)}y_1^{\drop(\pi)}b^{\fix(\pi)}(p+q)^{\cyc(\pi)-\fix(\pi)}.$$
Upon taking $a=IJ$ and $b=px_2+qy_2$ in the above expression, we get the desired result.
\end{proof}

Since the grammar~\eqref{G3} may be used in future, the associated labeling scheme is called $A$-labeling.
It is worth noting that Theorem~\ref{thm24} contains some interesting particular cases. 
\begin{corollary}\label{cor20}
We have 
\begin{itemize}
  \item When $x_1=y_1=1$, we have
  $$\sum_{\pi\in\ms_{n+1}}x^{\operatorname{pdes}(\pi)}y^{\operatorname{pasc}(\pi)}p^{\lrmin(\pi)-1}q^{\rlmin(\pi)-1}
=\sum_{\pi\in\msn}\left(px+qy\right)^{\fix(\pi)}(p+q)^{\cyc(\pi)-\fix(\pi)};$$
  \item When $x_2=y_2=1$, we have
  $$\sum_{\pi\in\ms_{n+1}}x^{\operatorname{impdes}(\pi)}y^{\operatorname{impasc}(\pi)}
p^{\lrmin(\pi)-1}q^{\rlmin(\pi)-1}
=\sum_{\pi\in\msn}x^{\exc(\pi)}y^{\drop(\pi)}(p+q)^{\cyc(\pi)};$$
 \item When $y_1=y_2=p=q=1$, we have
 $$\sum_{\pi\in\ms_{n+1}}x^{\operatorname{impdes}(\pi)}y^{\operatorname{pdes}(\pi)}
=\sum_{\pi\in\msn}x^{\exc(\pi)}(1+y)^{\fix(\pi)}2^{\cyc(\pi)-\fix(\pi)};$$
 \item When $y_1=y_2=p=q=1$ and $x_2=x_1=x$, we have
 $$A_{n+1}(x)=x\sum_{\pi\in\ms_{n+1}}x^{\operatorname{impdes}(\pi)+\operatorname{pdes}(\pi)}=\sum_{\pi\in\msn}x^{\exc(\pi)+1}(1+x)^{\fix(\pi)}2^{\cyc(\pi)-\fix(\pi)}.$$
\end{itemize}
\end{corollary}

Recall that $\mqn^{(1)}=\msn$. In the sequel, we give an extension of the $k=1$ case of Theorem~\ref{thm17}.
The number of {\it augmented proper ascents of $\pi$} is defined by
$$\operatorname{\widehat{pasc}}(\pi)=\left\{
  \begin{array}{ll}
    \operatorname{{pasc}}(\pi),~{\text{if $\pi(1)\neq 1$} };&  \\
   \operatorname{{pasc}}(\pi)+1,~{\text{if $\pi(1)=1$}}. &
  \end{array}
\right.$$
In other words, we say that an entry $\pi(i)$ of $\pi\in\msn$ is an {\it augmented proper ascent} if it
satisfies two conditions: $(i)$ $\pi(i-1)<\pi(i)$, 
where $i\in [n]$ and $\pi(0)=0$; $(ii)$ if $\pi(j)<\pi(i)$, then $j<i$, or $i=1$ and $\pi(1)=1$.
An {\it augmented improper ascent of $\pi$} is an entry $\pi(i)$ such that $\pi(i-1)<\pi(i)$, but $\pi(i)$ is not an augmented proper ascent, where $i\in[n]$ and we set $\pi(0)=0$.
Let $\operatorname{\widehat{impasc}}(\pi)$ be the number of augmented improper ascents of $\pi$. Then $$\operatorname{\widehat{impasc}}(\pi)=\asc(\pi)-\operatorname{\widehat{pasc}}(\pi).$$

\begin{example} 
We have
$$\operatorname{\widehat{pasc}}(\textbf{123})=3,~\operatorname{\widehat{pasc}}(\textbf{1}32)=1,~\operatorname{\widehat{pasc}}(21\textbf{3})=1,$$
$$\operatorname{\widehat{pasc}}(231)=0,~\operatorname{\widehat{pasc}}(31\textbf{2})=1,~\operatorname{\widehat{pasc}}(321)=0;$$
$$\operatorname{\widehat{impasc}}(123)=0,~\operatorname{\widehat{impasc}}(1\textbf{3}2)=1,~\operatorname{\widehat{impasc}}(\textbf{2}13)=1,$$
$$\operatorname{\widehat{impasc}}(\textbf{23}1)=2,~\operatorname{\widehat{impasc}}(\textbf{3}12)=1,~\operatorname{\widehat{impasc}}(\textbf{3}21)=1.$$
\end{example}
We end this section by giving the following result.
\begin{theorem}
We have 
$$\sum_{\pi\in\msn}x^{\operatorname{\widehat{impasc}}(\pi)}y^{\des^*(\pi)}p^{\operatorname{\widehat{pasc}}(\pi)}q^{\rlmin(\pi)}=
\sum_{\pi\in\msn}x^{\exc(\pi)}y^{\drop(\pi)}p^{\fix(\pi)}q^{\cyc(\pi)},$$
which implies that $(\operatorname{\widehat{impasc}},\des^*)$ is a symmetric distribution over $\msn$.
\end{theorem}
\begin{proof}
Let $\pi\in \msn$. We introduce a grammatical labeling of $\pi$:
\begin{itemize}
  \item [\rm ($L_1$)]If $\pi(i)$ is an augmented proper ascent, then put a superscript label $p$ just before $\pi(i)$;
  \item [\rm ($L_2$)]If $\pi(i)$ is an augmented improper ascent, then put a superscript label $x$ just before $\pi(i)$;
 \item [\rm ($L_3$)]If $\pi(i)$ is a descent, then put a superscript label $y$ right after $\pi(i)$;
 \item [\rm ($L_4$)]Put a superscript label $I$ right after $\pi$ and put a subscript label $q$ right after each right-to-left minimum.
\end{itemize}
When $n=1,2,3$, the labeled elements can be listed as follows:
$0^p1_q^I,~0^p1^p_q2_q^I,~0^x2^y1_q^I$,
$$0^p1^p_q2_q^I\rightarrow \left\{
  \begin{array}{ll}
   0^p1^p_q2^p_q3_q^I, &  \\
   0^p1_q^x3^y2_q^I, &\\
   0^x3^y1_q^p2_q^I;
  \end{array}
\right.~~0^x2^y1_q^I\rightarrow \left\{
  \begin{array}{ll}
   0^x2^y1_q^p3_q^I, &  \\
   0^x2^x3^y1^I_q, &\\
  0^x3^y2^y1^I_q.
  \end{array}
\right.$$
By induction, it is routine to verify that if 
$G=\{I\rightarrow Ipq, p\rightarrow xy, x\rightarrow xy, y\rightarrow xy\}$,
then $$D_G^n(I)=I\sum_{\pi\in\msn}x^{\operatorname{\widehat{impasc}}(\pi)}y^{\des^*(\pi)}p^{\operatorname{\widehat{pasc}}(\pi)}q^{\rlmin(\pi)},$$
where $\des^*(\pi)=\#\{i\in[n-1]: ~\pi(i)>\pi(i+1)\}$. By Lemma~\ref{lemmacycle}, we get the desired result. 
\end{proof}

By taking $x_2=y_2=w$, $x_1=x$ and $y_1=y$ in Theorem~\ref{thm24}, we obtain the following result.
\begin{corollary}We have
\begin{align*}
A_n\left(x,y,w,p+q\right)
&=\sum_{\pi\in\msn}x^{\operatorname{\widehat{impasc}}(\pi)}y^{\des^*(\pi)}w^{\operatorname{\widehat{pasc}}(\pi)}(p+q)^{\rlmin(\pi)}\\
&=\sum_{\pi\in\ms_{n+1}}x^{\operatorname{impdes}(\pi)}y^{\operatorname{impasc}(\pi)}w^{\operatorname{pdes}(\pi)+\operatorname{pasc}(\pi)}p^{\lrmin(\pi)-1}q^{\rlmin(\pi)-1}.
\end{align*}
In the case of $x=w=1$, we get 
$$\sum_{\pi\in\msn}y^{\des^*(\pi)}(p+q)^{\rlmin(\pi)}=\sum_{\pi\in\ms_{n+1}}y^{\operatorname{impasc}(\pi)}p^{\lrmin(\pi)-1}q^{\rlmin(\pi)-1}.$$

\end{corollary}

\section{Bivariate ascent-plateau polynomials}\label{section3}
\subsection{On the joint distribution of ascent-plateaux and left ascent-plateaux}
\hspace*{\parindent}


From~\eqref{Ankap}, we see that 
$$\sum_{\sigma\in\mqn}x^{\lap(\sigma)}=\sum_{\sigma\in\mqn}x^{n-\ap(\sigma)}.$$
Consider the bivariate polynomials 
$$P_n(x,y)=\sum_{\sigma\in\mqn}x^{\lap(\sigma)}y^{\ap(\sigma)}.$$
In particular, $P_1(x,y)=x$ and $P_2(x,y)=x+xy+x^2y$.
\begin{lemma}\label{lapap}
Let $G=\{I\rightarrow Jz,~J\rightarrow Jz^2+Ixz,~x\rightarrow 2xz^2,~z\rightarrow xz\}$.
Then we have $$D_G^n(I){\big{|}}_{\substack{I=1,~J=x,\\x=xy,~z=1}}=\sum_{\sigma\in\mqn}x^{\lap(\sigma)}y^{\ap(\sigma)}.$$
\end{lemma}
\begin{proof}
Given $\sigma\in\mqn$.  
We label $\sigma_i$ as follows:
\begin{itemize}
  \item [$(i)$] If $\sigma_1<\sigma_2$, we put a superscript label $I$ before $\sigma$. 
  If $\sigma_1=\sigma_2$, we label the two positions just before and right after $\sigma_1$ by a superscript label $J$; 
  \item [$(ii)$] If $\sigma_i$ is a ascent-plateau, then we label the two positions just before and right after $\sigma_i$ by a superscript label $x$, i.e., $\sigma_{i-1}\overbrace{\sigma_{i}}^x\sigma_{i+1}$, where $2\leqslant i\leqslant 2n-1$;
  \item [$(iii)$] Except the above two cases, we attach a superscript label $z$ to each of the other positions.
\end{itemize}
When $n=1,2$, the labeled elements are listed as follows:
$$\overbrace{1}^J1^z,~~\overbrace{2}^J2^z1^z1^z,~~^I1\overbrace{2}^x2^z1^z,~~\overbrace{1}^J1\overbrace{2}^x2^z.$$
Note that $D_G(I)=Jz$ and $D_G^2(I)=Jz(x+z^2)+Ixz^2$. Hence the results hold for $n=1,2$. 
We proceed by induction. Now assume that $m\geqslant 2$ and $\sigma\in\mq_{m-1}$. Suppose that 
$\sigma'$ is Stirling permutation in $\mq_m$ created
from $\sigma$ by inserting the string $mm$.
Then the changes of labeling can be illustrated as follows:
$$^I\sigma\mapsto \overbrace{m}^Jm^z\sigma;\quad\quad~~~\cdots\sigma_i^z\sigma_{i+1}\cdots\mapsto \cdots\sigma_i\overbrace{m}^xm^z\sigma_{i+1};$$
$$\overbrace{\sigma_1}^J\sigma_2\cdots\mapsto  \overbrace{m}^Jm^z\sigma_1^z\sigma_2\cdots~{\text{or}}~^I\sigma_1\overbrace{m}^xm^z\sigma_2\cdots;$$
$$\cdots\overbrace{\sigma_{i}}^x\sigma_{i+1}\cdots\mapsto \cdots \overbrace{m}^xm^z{\sigma_{i}}^z\sigma_{i+1}\cdots~{\text{or}}~\cdots^z{\sigma_{i}}\overbrace{m}^xm^z\sigma_{i+1}\cdots.$$
Summing up all the cases shows that the assertion is valid for $m$ and the insertion of the string $mm$ corresponds to the operator $D_G$.
This completes the proof.
\end{proof}

We now give a unified extension of the bi-$\gamma$-positivity of $M_n(x)$ and $N_n(x)$. 
\begin{theorem}
We have 
$$\sum_{\sigma\in\mq_{n+1}}x^{\lap(\sigma)}y^{\ap(\sigma)}=x\sum_{i\geqslant 0}\xi_{n,i}(xy)^i(1+xy)^{n-2i}+xy\sum_{j\geqslant 0}\eta_{n,j}(xy)^j(1+xy)^{n-1-2j},$$
where the coefficients $\xi_{n,i}$ and $\eta_{n,j}$ satisfy the recurrence system 
 \begin{equation}\label{xieta}
\left\{
  \begin{array}{ll}
   \xi_{n+1,i}=(1+2i)\xi_{n,i}+4(n-2i+2)\xi_{n,i-1}+\eta_{n,j-1}, & \\
\eta_{n+1,j}=(2+2j)\eta_{n,j}+4(n-2j+1)\eta_{n,j-1}+\xi_{n,j}, &
  \end{array}
\right.
\end{equation}
with the initial conditions $\xi_{1,0}=\eta_{1,0}=1$ and $\xi_{1,k}=\eta_{1,k}=0$ for $k>0$.
\end{theorem}
\begin{proof}
Let $G=\{I\rightarrow Jz,~J\rightarrow Jz^2+Ixz,~x\rightarrow 2xz^2,~z\rightarrow xz\}$. Note that 
$D_G(I)=Jz,~D_G^2(I)=Jz(x+z^2)+Ixz^2$. Set $u=xz^2$ and $v=x+z^2$. We find that $D_G(u)=2uv$ and $D_G(v)=4u$. So we get
$D_G(I)=Jz,~D_G(Jz)=Jzv+Iu,~D_G^2(Jz)=Jz(v^2+5u)+Iu(3v)$.
Assume that 
\begin{equation}\label{DGIJ}
D_G^{n+1}(I)=D_G^{n}(Jz)=Jz\sum_{i\geqslant 0}\xi_{n,i}u^iv^{n-2i}+Iu\sum_{j\geqslant 0}\eta_{n,j}u^jv^{n-1-2j}.
\end{equation}
We proceed by induction. Note that 
\begin{align*}
D_G^{n+1}(I)&=D_G\left(Jz\sum_{i\geqslant 0}\xi_{n,i}u^iv^{n-2i}+Iu\sum_{j\geqslant 0}\eta_{n,j}u^jv^{n-1-2j}\right)\\
&=(Jzv+Iu)\left(\sum_{i\geqslant 0}\xi_{n,i}u^iv^{n-2i}\right)+Jz\sum_{i}\xi_{n,i}\left(2iu^{i}v^{n+1-2i}+4(n-2i)u^{i+1}v^{n-1-2i}\right)+\\
&(Jzu+2Iuv)\left(\sum_{j\geqslant 0}\eta_{n,j}u^jv^{n-1-2j}\right)+Iu\sum_{j}\eta_{n,j}\left(2ju^jv^{n-2j}+4(n-1-2j)u^{j+1}v^{n-2-2j}\right).
\end{align*}
Comparing the coefficients, we get~\eqref{xieta}.
By Lemma~\ref{lapap}, substituting $J\rightarrow x,I\rightarrow 1,~z\rightarrow 1$ and $x\rightarrow xy$, then $u=xy$ and $v=1+xy$, which yields  
the desired symmetric decomposition.
\end{proof}

Combining~\eqref{DGIJ} and Lemma~\ref{lapap}, we get the following.
\begin{corollary}
We have 
$$\sum_{\substack{\sigma\in\mq_{n+1}\\\sigma_1=\sigma_2}}y^{\ap(\sigma)}=\sum_{i\geqslant 0}\xi_{n,i}y^i(1+y)^{n-2i},~\sum_{\substack{\sigma\in\mq_{n+1}\\\sigma_1<\sigma_2}}y^{\ap(\sigma)}=y\sum_{j\geqslant 0}\eta_{n,j}y^{j}(1+y)^{n-1-2j}.$$
\end{corollary} 

Define $\xi_n(x)=\sum_{i\geqslant 0}\xi_{n,i}x^i$ and  $\eta_n(x)=\sum_{j\geqslant 0}\eta_{n,j}x^j$. It follows from~\eqref{xieta} that 
 \begin{equation*}
\left\{
  \begin{array}{ll}
   \xi_{n+1}(x)=(1+4nx)\xi_n(x)+2x(1-4x)\frac{\mathrm{d}}{\mathrm{d}x}\xi_n(x)+x\eta_n(x), & \\
\eta_{n+1}(x)=(2+4nx-4x)\eta_{n}(x)+2x(1-4x)\frac{\mathrm{d}}{\mathrm{d}x}\eta_n(x)+\xi_n(x), &
  \end{array}
\right.
\end{equation*}
with $\xi_1(x)=\eta_1(x)=1,~\xi_2(x)=1+5x,~\xi_3(x)=1+26x,~\eta_2(x)=3,~\eta_3(x)=7+17x$.
Set $F_n(x)=\xi_n(x^2)+x\eta_n(x^2)$. Using~\cite[Theorem~3]{Ma22}, we arrive at
$$F_{n+1}(x)=(1+x+4nx^2)F_n(x)+x(1-4x^2)\frac{\mathrm{d}}{\mathrm{d}x}F_n(x),$$
with the initial conditions $F_0(x)=1$ and $F_1(x)=1+x$.
\subsection{Euler-Stirling statistics on Stirling permutations}
\hspace*{\parindent}

Let $\left(a_n(x),b_n(x)\right)$ be the symmetric decomposition of $M_n(x)$. Following~\cite[Proposition~3.15]{Ma24}, we see that
\begin{equation}\label{Stirling-recu4}
\left\{
  \begin{array}{ll}
a_{n+1}(x)&=(1+x+2(n-1)x)a_{n}(x)+2x(1-x)\frac{\mathrm{d}}{\mathrm{d}x}a_{n}(x)+xb_{n}(x),\\
b_{n+1}(x)&=2(1+(n-1)x)b_{n}(x)+2x(1-x)\frac{\mathrm{d}}{\mathrm{d}x}b_{n}(x)+a_{n}(x),
  \end{array}
\right.
\end{equation}
with $a_{1}(x)=1$ and $b_{1}(x)=0$. 
In this following, we consider the {\it $q$-ascent-plateau polynomials}
$$M_n(x,q)=\sum_{\sigma\in\mqn}x^{\ap(\sigma)}q^{\lrmin(\sigma)}.$$
For instance, $M_1(x,q)=q,~M_2(x,q)=q^2+2qx$.

For $\sigma\in\mqn^{(k)}$, we say that an value $\sigma_i$ is 
a {\it shortest ascent-plateau} if $\sigma_{i-1}<\sigma_{i}=\sigma_{i+1}$, where $2\leqslant i\leqslant nk-1$.
Let $\operatorname{ap}_2(\sigma)$ denote the number of shortest ascent-plateaux of $\sigma$. When $\sigma\in\mqn$, we have 
$\operatorname{ap}_2(\sigma)=\ap(\sigma)$.
\begin{lemma}\label{lemma1}
Let $G=\{I\rightarrow qIy^k,~x\rightarrow 2xy^k,~y\rightarrow xy^{k-1}\}$. Then 
$$D_G^n(I)=I\sum_{\sigma\in\mqn^{(k)}}x^{\operatorname{ap}_2(\sigma)}q^{\lrmin(\sigma)}y^{kn-2\operatorname{ap}_2(\sigma)}.$$
\end{lemma}
\begin{proof}
Given $\sigma\in\mqn^{(k)}$. 
We first attach a subscript label $q$ to each left-to-right minimum. Then the
label scheme is given by the following procedure:
\begin{itemize}
  \item [$(i)$] Put a label $I$ at the front of $\sigma$; 
  \item [$(ii)$] If $\sigma_i$ is a shortest ascent-plateau, then we label the two positions just before and right after $\sigma_i$ by a superscript label $x$, i.e., $\sigma_{i-1}\overbrace{\sigma_{i}}^x\sigma_{i+1}\cdots\sigma_{i+k-1}$, where $2\leqslant i\leqslant nk-k+1$;
  \item [$(iii)$] Except the above two cases, we attach a superscript label $y$ to each of the other positions.
\end{itemize}
Note that the weight of $\sigma$ is given by $w(\sigma)=Ix^{\operatorname{ap}_2(\sigma)}q^{\lrmin(\sigma)}y^{kn-2\operatorname{ap}_2(\sigma)}$.
For example,
$$^I5_q^y5^y5^y1^y_q1^y2^y2\overbrace{3}^x3^y3^y2^y1^y4\overbrace{6}^x6^y6^y4^y4^y.$$
Note that 
$\mq_1^{(k)}=\left\{^I1_q^y1^y\cdots 1^y\right\}$. 
So the weight of $11\cdots1$ is given by $D_G(I)$. We proceed by induction. Now assume that $m\geqslant 2$ and $\sigma\in\mq_{m-1}^{(k)}$. Suppose that 
$\sigma'$ is Stirling permutation in $\mq_m^{(k)}$ created
from $\sigma$ by inserting the string $mm\cdots m$.
Then the changes of labeling can be illustrated as follows:
$$^I\sigma\mapsto ~^Im_q^ym^y\cdots m^y\sigma; \quad\quad~~~~\cdots\sigma_i^y\sigma_{i+1}\cdots\mapsto \cdots\cdots\sigma_i \overbrace{m}^xm^ym^y\cdots m^y\sigma_{i+1}\cdots;$$
$$\cdots\overbrace{\sigma_{i}}^x\cdots\mapsto \cdots\overbrace{m}^xm^ym^y\cdots m^y\sigma_i^y\cdots~{\text{or}}~\cdots^y\sigma_i\overbrace{m}^xm^ym^y\cdots m^y\cdots.$$
Summing up all the cases shows that the assertion is valid for $m$ and the insertion of the string $mm\cdots m$ corresponds to the operator $D_G$.
This completes the proof.
\end{proof}

We can now conclude the following result. The $q=1$ case was first established in~\cite{Ma24} and then investigated by Yan-Yang-Lin~\cite{Yan26}.
\begin{theorem}\label{thm2}
For $n\geqslant 2$, the polynomials $M_n(x,q)$ are bi-$\gamma$-positive if $0<q<2$.
\end{theorem}
\begin{proof}
Let $D_G$ be the formal derivative with respect to $G=\{I\rightarrow qIy^2,~x\rightarrow 2xy^2,~y\rightarrow xy\}$.
Then $D_G(I)=qIy^2$, $D_G(x)=2xy^2$ and $D_G(y)=xy$.
By Lemma~\ref{lemma1}, we see that $$D_G^n(I)=I\sum_{\sigma\in\mqn}x^{\ap(\sigma)}q^{\lrmin(\sigma)}y^{2n-2\ap(\sigma)}.$$
Note that 
$D_G^2(I)=q^2Iy^{4}+2qIxy^{2}=q^2Iy^2(x+y^2)+qIxy^2(2-q)$.
Setting $A=Iy^2$ and $v=x+y^2$, we obtain  
$$D_G(I)=qA,~D_G(A)=qIy^2(x+y^2)+Ixy^2(2-q)=qAv+(2-q)Ax=A\left(qv+(2-q)x\right).$$
Furthermore, setting $u=xy^2$, we have $x^2=xv-u$.
So we have
\begin{align*}
D_G^2(A)&=A\left(qv+(2-q)x\right)^2+A(4qxy^2+2(2-q)xy^2)\\
&=A(q^2v^2+4qu)+A((2-q)^2x^2+2q(2-q)xv+2(2-q)u)\\
&=A(q^2v^2+4qu)+A((2-q)^2(xv-u)+2q(2-q)xv+2(2-q)u)\\
&=A(q^2v^2+4qu)+A((2q-q^2)u+(2-q)(1+2q)xv)\\
&=A(q^2v^2+(6q-q^2)u)+Ax((2-q)(2+q)v).
\end{align*}
Note that $D_G(u)=2uv$ and $D_G(v)=4u$. We get a new grammar $$H=\{A\rightarrow A\left(qv+(2-q)x\right), x\rightarrow 2u,~u\rightarrow 2uv,v\rightarrow4u\}.$$
Suppose that 
$D_H^n(A)=Af_n(u,v)+Axg_n(u,v)$.
Then we get 
\begin{align*}
D_H^{n+1}(A)&=A(qv+(2-q)x)(f_n(u,v)+xg_n(u,v))+\\
&A(D_H(f_n(u,v))+2ug_n(u,v)+xD_H(g_n(u,v))\\
&=A(qvf_n(u,v)+D_H(f_n(u,v))+2ug_n(u,v))+\\
&A((2-q)xf_n(u,v)+qvxg_n(u,v)+(2-q)x^2g_n(u,v)+xD_H(g_n(u,v))).
\end{align*}
Since $x^2=xv-u$, so $(2-q)x^2g_n(u,v)=(2-q)xvg_n(u,v)-(2-q)ug_n(u,v)$.
We arrive at 
\begin{align*}
D_H^{n+1}(A)&=A(qvf_n(u,v)+D_H(f_n(u,v))+qug_n(u,v))+\\
&Ax((2-q)f_n(u,v)+2vg_n(u,v)+D_H(g_n(u,v)).
\end{align*}
Therefore, we get 
\begin{equation}\label{recusystem}
\left\{
  \begin{array}{ll}
 f_{n+1}(u,v)&=qvf_n(u,v)+D_H(f_n(u,v))+qug_n(u,v),\\
 g_{n+1}(u,v)&=(2-q)f_n(u,v)+2vg_n(u,v)+D_H(g_n(u,v)),
  \end{array}
\right.
\end{equation}
with the initial conditions $f_1(u,v)=qv$ and $g_1(u,v)=2-q$.
When $0<q\leqslant 2$, suppose that there exist nonnegative real numbers $f_{n,i}(q)$ and $g_{n,j}(q)$ such that 
$$f_n(u,v)=\sum_{i\geqslant 0}f_{n,i}(q)u^iv^{n-2i},~g_n(u,v)=\sum_{j\geqslant 0}g_{n,j}(q)u^jv^{n-1-2j}.$$
It follows from~\eqref{recusystem} that 
\begin{equation}\label{recusystem02}
\left\{
  \begin{array}{ll}
 f_{n+1,i}(q)&=(q+2i)f_{n,i}(q)+4(n-2i+2)f_{n,i-1}(q)+qg_{n,i-1}(q);\\
 g_{n+1,i}(q)&=2(1+i)g_{n,i}(q)+4(n-2i+1)g_{n,i-1}(q)+(2-q)f_{n,i}(q),
  \end{array}
\right.
\end{equation}
with $f_{1,0}(q)=q$, $g_{1,0}(q)=2-q$ and $f_{1,i}(q)=g_{1,i}(q)=0$ if $i>0$. 
Clearly, both $f_{n,i}(q)$ and $g_{n,j}(q)$ are nonnegative integers when $0<q\leqslant 2$.
In conclusion, we get
\begin{equation}\label{DGA2}
D_G^n(A)=A\sum_{i\geqslant 0}f_{n,i}(q)u^iv^{n-2i}+Ax\sum_{j\geqslant 0}g_{n,j}(q)u^jv^{n-1-2j}.
\end{equation}
By taking $I=y=1$, we get $M_n(x,q)=D_G^n(I)|_{I=y=1}=qD_H^{n-1}(A)|_{A=1,u=x,v=1+x}$. So we get 
$$M_n(x,q)=\sum_{i\geqslant 0}qf_{n-1,i}(q)x^i(1+x)^{n-1-2i}+x\sum_{j\geqslant 0}qg_{n-1,j}(q)x^j(1+x)^{n-2-2j},$$
as desired. So we finish the proof.
\end{proof}
For example, we have 
$$M_3(x,q)=q^3 + 4q x + 6 q^2 x + 4 qx^2=(q^3(1+x)^2+(6q-q^2)qx)+q(2-q)(2+q)x(1+x).$$
\begin{corollary}
For $n\geqslant 2$, the polynomials $M_n(x,q)$ are alternatingly increasing if $0<q<2$.
\end{corollary}
When $q=2$, it follows from~\eqref{recusystem02} that
$$x^nM_n(1/x,2)=2x\sum_{i\geqslant 0}f_{n-1,i}(2)x^i(1+x)^{n-1-2i}=2\sum_{j\geqslant 1}f_{n-1,j-1}(2)x^j(1+x)^{n+1-2j},$$
where the coefficients $f_{n-1,i}(2)$ satisfy the recursion 
\begin{equation}\label{ank-recu02}
f_{n+1,i}(2)=(2+2i)f_{n,i}(2)+4(n-2i+2)f_{n,i-1}(2).
\end{equation}
Comparing~\eqref{ank-recu} with~\eqref{ank-recu02}, we find that $2^nA_n(x)=x^nM_n(1/x,2)$.

\section*{Acknowledgements}
Shi-Mei Ma was supported by the National Natural Science Foundation of China (No. 12071063) and
Taishan Scholars Foundation of Shandong Province (No. tsqn202211146). 
Jianfeng Wang was supported by National Natural Science Foundation of China (No. 12371353).
Guiying Yan was supported by the National Natural Science Foundation of China (No. 12231018).
Jean Yeh was supported by the National science and technology council (Grant number: MOST 1132115M017005MY2).
\bibliographystyle{amsplain}

\end{document}